\documentclass{ifacconf}

\usepackage{graphicx}      
\usepackage{natbib}        

\usepackage{cuted}
\usepackage{graphics} 
\usepackage{epsfig} 
\usepackage{mathptmx} 
\usepackage{times} 
\usepackage{amsmath} 
\usepackage{amssymb} 
\usepackage{booktabs} \usepackage{multirow} 
\usepackage{url} 
\DeclareMathAlphabet{\mathcal}{OMS}{cmsy}{m}{n} 
\usepackage{algorithm}
\usepackage{algpseudocode}
\newtheorem{assumptions}{Assumptions}
\newtheorem{theorem}{Theorem}

\newtheorem{definition}{Definition}

\newenvironment{proof}{%
  \noindent \textit{Proof:} \quad
}{
  \hfill $\square$ \vspace{1ex} 
}

\usepackage{tikz} 
\usepackage{tikz-cd}
\usetikzlibrary{positioning, shapes.geometric, arrows.meta}
\usetikzlibrary{matrix, positioning, shapes.geometric, arrows} 
\tikzstyle{block} = [rectangle, draw, rounded corners, minimum width=3cm, minimum height=1.2cm, align=center] 
\tikzstyle{data} = [rectangle, draw, minimum width=3cm, minimum height=1.2cm, align=center] \tikzstyle{arrow} = [thick, ->, >=stealth]
\usepackage{subcaption}
\begin{document}
\begin{frontmatter}
\title{A PAC-Bayes Approach for Controlling Unknown Linear Discrete-time Systems}
\author[First]{Yujia Luo}
\author[First]{Ye Pu}
\author[First]{Jonathan H. Manton}
\author[First]{Jingge Zhu}

\address[First]{
The University of Melbourne, Australia.\\
Emails: \{yujia.luo.1, ye.pu, jmanton, jingge.zhu\}@unimelb.edu.au
}

\begin{abstract}
This paper presents a PAC-Bayes framework for learning controllers for unknown stochastic linear discrete-time systems, where the system parameters are drawn from a fixed but unknown distribution. We derive a data-dependent high-probability bound on the performance of any learned (stochastic) controller, and propose novel efficient learning algorithms with theoretical guarantees, which can be implemented for both finite and infinite controller spaces. Compared to prior work, our bound holds for unbounded quadratic cost. In the special case where LQG is optimal, our numerical results suggest that the learned controllers achieve comparable performance to LQG. 
\end{abstract}


\end{frontmatter}

\section{Introduction}
\label{sec:intro}
Controlling linear discrete-time systems is a fundamental engineering problem. 
For known linear dynamics with quadratic costs and Gaussian noise, classical LQG control provides an optimal solution~\citep{lqgcases}. 
For unknown dynamics, learning-based control methods are required.
A classical approach is adaptive LQG, which combines online parameter estimation with stochastic control design and achieves asymptotically optimal performance for time-invariant systems without requiring an accurate prior model~\citep{adaptivelqg,adaptivelqg2,adaptivelqg3}. Related research on model uncertainty includes backstepping control~\citep{Backstepping}, robust learning-based methods~\citep{robust}, dual control strategies~\citep{dual}, and reinforcement learning-based algorithms~\citep{rl}.

In contrast to the classical assumption of fixed but unknown parameters, this work considers system dynamics modeled as random variables drawn from unknown but fixed distributions. 
This setting arises, for instance, when multiple systems are produced under the same manufacturing process, where each instance has parameter variations due to process uncertainty while following a common underlying distribution.
The PAC-Bayes framework offers opportunities to achieve this goal by providing performance guarantees for learned controllers through combining a data-independent distribution with empirical data.
It gives a theoretical guarantee for the performance of the learned policy on unseen data through a high-probability upper bound on the expected cost~\citep{c1}.

\citep{c4,c5} have shown that the PAC--Bayes framework can be successfully integrated into control design for known or well-structured systems. 
In particular, \citep{c4} used PAC-Bayes to optimize robotic control strategies with improved generalization across unseen environmental conditions.
This PAC-Bayes-based method has subsequently been explored in co-adaptive human–robot interaction~\citep{paclearning1} and, more recently, in combination with conformal prediction for formal verification and safe control of learning-enabled autonomous systems~\citep{paclearning2}.
For nonlinear systems subject to unknown noise disturbances, \citep{c5} developed a PAC-Bayes-based method to maintain reliable control performance, in which a $\tanh$ transformation is used to map unbounded costs into a bounded range so that the classical PAC-Bayes bound can be applied.

The PAC--Bayes framework still has two gaps in control applications. 
The first gap concerns handling unbounded quadratic costs through PAC-Bayes bounds that directly accommodate this cost structure and preserve the stronger penalization of poor controllers.
The second gap concerns unknown systems, in which the data-only nature of PAC-Bayes can be used to learn a control policy directly from data without an explicit model.

\subsection{Contributions}
\begin{itemize}   
    \item PAC-Bayes bounds are derived for controller learning over finite and infinite controller spaces. The bounds apply directly to the original unscaled quadratic cost, without boundedness or normalization.
   
    \item A PAC-Bayes-based controller learning algorithm is proposed for arbitrary unknown distributions over the system matrices and process noise. Under additional structural and boundedness assumptions, it yields computable high-probability upper bounds on the expected quadratic cost.
    
    \item Numerical experiments on unknown linear discrete-time systems show effective learned policies, with expected costs comparable to an optimal LQG benchmark computed using the true system matrices.
\end{itemize}

\subsection{Notation}
Let $\mathcal{D}^m$ be the $m$-fold Cartesian product of $\mathcal{D}$, and let $KL(P\|P_0):=\mathbb{E}_{x\sim P}\ln\frac{P(x)}{P_0(x)}$. 
For $x\in\mathbb{R}^d$, $\|x\|_2$ denotes the Euclidean norm. 
For $A\in\mathbb{R}^{m\times n}$, $\|A\|$ and $\|A\|_F$ denote the spectral and Frobenius norms. 
For square $A$, $r_A:=\max_i|\lambda_i|$, where $\lambda_i$ are its eigenvalues. 
When $A$ is symmetric or positive definite, $\|A\|=r_A$, so $\|A\|$ coincides with the spectral radius.

\section{Problem Formulation}
\label{sec:problem}

\subsection{System model and quadratic cost}
Consider the time-invariant discrete-time linear system:
\begin{equation}\label{eq:Time-invariant Linear Discrete-time System}
    x{(t+1)} = A x{(t)} + B u{(t)} + w{(t)}, \quad t= 0,\cdots,T-1, 
\end{equation}
where $x{(t)} \in \mathbb{R}^{d_x}$ is the state vector. The initial state $x{(0)}$ is assumed to be $\mathbf{0}$ in this paper without loss of generality. Here, $u(t)\in\mathbb{R}^{d_u}$ denotes the control input, and $T$ is the fixed time horizon. The system parameters $A\in \mathbb{R}^{{d_x}\times{d_x}}$ and $B\in \mathbb{R}^{{d_x}\times{d_u}}$ are random variables drawn from unknown but fixed distributions $\mathcal{D}_A$, $\mathcal{D}_B$, respectively, and their realizations are not available for observation. The sequence $\{w{(t)}\}_{t=0}^{T-1}$ with $w(t)\in \mathbb{R}^{d_x}$ represents process noise, drawn i.i.d. from the unknown but fixed distribution $\mathcal{D}_w$. In Section~\ref{sec:thm}, we introduce additional boundedness and tail assumptions on $(\mathcal D_A,\mathcal D_B,\mathcal D_w)$ that are needed for theoretical analysis. The state-feedback control law is given by:
\begin{equation}\label{eq:Time-invariant control law}
    u{(t)} = K x{(t)},
\end{equation}
where $K\in\mathcal{K}\subseteq\mathbb{R}^{d_u \times d_x }$ is the control matrix that determines the control input based on the system state, and $\mathcal{K}$ denotes the controller space. 

Applying a specific controller $K$ to a randomly chosen system produces a random trajectory $x{(1)},\cdots,x{(T)}$, which we denote by the random variable $X_K = (x{(1)},\cdots,x{(T)})\in \mathbb{R}^{{d_x}\times T}$. Specifically, $X_K$ is obtained by sampling $A \sim \mathcal{D}_A$, $B \sim \mathcal{D}_B$, and a noise sequence $w(t) \overset{\text{i.i.d.}}{\sim} \mathcal{D}_w$ for $t=0,\dots,T-1$, and evolving the system~\eqref{eq:Time-invariant Linear Discrete-time System}-\eqref{eq:Time-invariant control law}.
This procedure induces a probability measure over trajectories, denoted by $P_{X_K}$.

In a typical LQG framework, the cost is defined as
\begin{align}\label{eq:J_KAB}
    J(K;A,B)
    := & \mathbb{E}_{W}\Bigg[
        \sum_{t=0}^{T-1} \bigl( x(t)^\top Q x(t) + u(t)^\top R u(t) \bigr)\notag\\        
        &+ x(T)^\top Q x(T)
    \Bigg],
\end{align}
where the expectation is over the noise sequence $W :=(w{(0)},\cdots,w{(T-1)})$, and $Q \succeq 0$ and $R \succ 0$ weight the state and control input. 
Since $A$ and $B$ are random, we also take expectations over them, that is,
$
    C(K) := \mathbb{E}_{A \in \mathcal{D}_A,\; B \in \mathcal{D}_B} J(K;A,B),
$
which can be rewritten as
\begin{equation}\label{eq:C_over_XK}
    C(K) = \mathbb{E}_{X_K \sim P_{X_K}} \bigl[ C_\text{q}(K,X_K) \bigr],
\end{equation}
where we define 
\begin{align}\label{eq: quadratic cost}
    C_\text{q}(K,X_K) &:= \sum_{t=0}^{T-1}  \left[ x{(t)}^\top Q x{(t)} + (Kx{(t)})^\top R (Kx{(t)}) \right]\notag\\& + x{(T)}^\top Q x{(T)}
\end{align} 
with the understanding that the randomness of $X_K$ is induced by random $A$, $B$, and the noise sequence as described above.

In this work, we consider a stochastic learning algorithm that maps the
training dataset to a learned controller $K$ through a
data-dependent randomized procedure. Equivalently, the algorithm first outputs a data-dependent distribution $P$ over the controller space $\mathcal K$, from which the learned controller is selected as $K\sim P$.
Therefore, the ultimate cost is given by taking the expected value of $C(K)$ with respect to $K \sim P$:
\begin{equation}\label{eq:ultimate_cost}
    \bar C(P) := \mathbb{E}_{K \sim P} \bigl[ C(K) \bigr].
\end{equation}
In the PAC--Bayes literature, $\bar C(P)$ is termed the \emph{Gibbs expected cost}~\citep{c3}.

\subsection{Preliminaries on PAC-Bayes approach}

In the PAC-Bayes framework, the learning object is a probability distribution $P$ over a hypothesis space. 
In this paper, the hypothesis space is the controller space $\mathcal K$, and each hypothesis is a feedback gain $K \in \mathcal K$. 
To apply the PAC-Bayes approach, we first choose a fixed, data-independent reference distribution $P_0$ on $\mathcal K$, and then learn a data-dependent distribution $P$ from data. 
In the PAC-Bayes literature, $P_0$ and $P$ are commonly referred to as the \textit{prior distribution} and \textit{posterior distribution}, respectively~\citep{c1}. 
We restrict attention to distributions $P$ for which $KL(P\|P_0) < \infty$, which in particular requires $P$ to be absolutely continuous with respect to $P_0$.
PAC-Bayes theory provides high-probability performance guarantees by relating the expected cost under $P$ to the empirical cost observed from data, together with a complexity term involving $KL(P\|P_0)$.

\subsection{Control objectives}
We consider two cases in our study:
\begin{itemize}
    \item \textbf{Case 1 (Finite controller space).}  
    A finite set of controllers is given by  
    $\mathcal{K} = \{K_1, K_2, \ldots, K_L\}.$

    \item \textbf{Case 2 (Infinite controller space).} 
    An infinite controller set is considered, satisfying  
    $\mathcal{K} \subseteq \mathbb{R}^{d_u \times d_x}.$
\end{itemize}

Our objective is to learn a controller from trajectory data generated by unknown systems.
We consider a stochastic learning algorithm that maps the training dataset $\mathcal S$ to a learned controller $K$.
Within the PAC--Bayes framework, this procedure is characterized by a data-dependent distribution $P$ over $\mathcal K$.

A natural objective is to choose $P$ to minimize the expected cost:
$
\min_{P} \, \bar C(P).
$
However, since $\mathcal D_A$, $\mathcal D_B$, and $\mathcal D_w$ are unknown, $\bar C(P)$ is not computable and therefore cannot be optimized directly.

The PAC--Bayes approach serves two main purposes.
\begin{itemize}
    \item \textbf{Objective 1 (theoretical guarantees):}  
    To provide a computable and data-dependent upper bound on $\bar{C}(P)$ that certifies the performance of any given control strategy, as discussed in Section~\ref{sec:thm}.

    \item \textbf{Objective 2 (controller learning algorithm):}
    Based on the derived PAC--Bayes bound, we aim to develop a data-driven learning objective which, given only a finite set of observed trajectories, learns a distribution $P$ over controllers, as discussed in Section~\ref{sec:algorithm}.
\end{itemize}

\section{Two PAC-Bayes Bounds for Controlling Unknown Systems}
\label{sec:thm}

In this section, we address Objective~1 by first introducing the definition of sub-Gaussian random variables.
\begin{definition}[Sub-Gaussian random variable with constant $\sigma^2$]~\citep{subgaussian}\label{def:sub-Gaussian}
A real-valued random variable $V$ is called $\sigma^2$-sub-Gaussian if there exists a constant $\sigma > 0$ such that for all $\lambda \in \mathbb{R}$,
\begin{equation*}
    \mathbb{E}\bigl(e^{\lambda [V-\mathbb{E}(V)]}\bigr)
    \le e^{\sigma ^2 \lambda^2/2}.
\end{equation*}
\end{definition}

\begin{definition}[Sub-Gaussian random vector with constant $\sigma^2$]\label{def:sub-Gaussian vector}
Let $\sigma > 0$. A random vector $V = (V_{1}, \ldots, V_{d}) \in \mathbb{R}^d$
is called $\sigma^2$-sub-Gaussian if for any unit vector
$U \in \mathbb{R}^d$ with $\|U\|_2 = 1$, the random variable $U^\top V$
is $\sigma^2$-sub-Gaussian.
\end{definition}

\vspace*{-1.5em}
\begin{strip}
\begin{equation}
    \label{eq:Bcost}
B_{\text{cost}}
:= \operatorname*{ess\,sup}_{K\in\mathcal K}\;
\sqrt{%
\sigma_w^{4} d_x \,\rho_Z(K)^2
\left(
128 \sum_{i=0}^{T} \rho_M(K)^{\,4(T-i)} \right.
\left. +\; 64 \sum_{0\le i<j\le T} 
\rho_M(K)^{\,4T-2(i+j)} \right)}.
\end{equation}
\end{strip}
\vspace*{-2em}
\noindent
\subsection{Data-generating process and required assumptions}

\textbf{Data-generating process.} For a fixed controller $K$, we assume that $n$ independent realizations $\{(A^{(i)},B^{(i)})\}_{i=1}^n$ are obtained by sampling from unknown distributions $\mathcal{D}_A$ and $\mathcal{D}_B$. Similarly, we assume the noise terms $\{w^{(i)}{(t)}\}^{i=1,\cdots,n}_{t=0,\cdots,T-1}$ are i.i.d. with fixed but unknown distribution $\mathcal D_w$. Then $n$ samples of the trajectories $\{X_K^{(i)}\}_{i=1}^n$ can be generated, where each matrix $X_K^{(i)} = (x^{(i)}{(1)},\cdots,x^{(i)}{(T)})\in \mathbb{R}^{{d_x}\times T}$ represents one trajectory. We define the empirical quadratic cost for a controller $K$ as
\begin{equation}\label{eq:
empirical cost}
    \hat{C}(K) := \frac{1}{n}\sum_{i=1}^n C_q (K,X_K^{(i)}).
\end{equation}

\begin{assumptions}\label{assumption:sub-Gaussian-TI}
There exist known constants $a_1,a_2,b_1,b_2 \in \mathbb{R}$ and $\sigma_w^2>0$, such that:
\begin{itemize}
    \item $A \sim \mathcal  D_A$, $B \sim \mathcal  D_B$ and
    \[
        \mathbb{P}\bigl(a_1 \le A_{ij} \le a_2\bigr) = 1,\quad
        \mathbb{P}\bigl(b_1 \le B_{ij} \le b_2\bigr) = 1,
        \quad \forall\, i,j.
    \]
    \item $W :=(w{(0)},\cdots,w{(T-1)})$ is an i.i.d.\ sequence with $w(t) \sim \mathcal D_w$, where
    $
        w(t) = (w_1(t),\ldots,w_{d_x}(t))^\top,
    $
    and each $w_k(t)$ is $\sigma_w^2$-sub-Gaussian.
\end{itemize}
\end{assumptions}

\subsection{PAC-Bayes bound for a finite controller space}
For a finite controller space $\mathcal K=\{K_1,\cdots,K_L\}$, the dataset $\mathcal{S} =\{X_{K_j}^{(i)}\}^{i=1,\cdots ,n}_{j=1,\cdots,L}$ is generated according to the data-generating process described above, and when Assumptions~\ref{assumption:sub-Gaussian-TI} hold, each trajectory realization ${X_{K_j}^{(i)}}$ is i.i.d. Let $P$ be a vector representing a probability mass function on $\mathcal{K}$ with coordinates $P_j=P(K_j)$. When the distribution depends on the data $\mathcal S$ (i.e., it is a conditional distribution), it can be interpreted as a learning algorithm which stochastically maps data to a controller. That is, given the data $\mathcal S$ as input, the learning algorithm specifies a random controller according to the distribution $P$. In Section~\ref{sec:algorithm} we propose algorithms for obtaining $P$.

Define the \textit{Gibbs empirical cost} with probability distribution $P$ as
$\tilde{C}(P):=\mathbb{E}_{K\sim P}[\hat C(K)]=\sum_{j=1}^L P_j\,\hat C(K_j),$
where $\hat C(K)$ is defined in Eq.~\eqref{eq:
empirical cost}.  In this case, the expression can be explicitly written as 
$\tilde{C}(P)=\sum_{j=1}^L P_j\,[\frac{1}{n}\sum_{i=1}^n C_q (K_j,X_{K_j}^{(i)})].$
We now present the following theorem.

\begin{theorem}[PAC-Bayes bound for finite controller space]
\label{theorem: PAC-Bayes Bound for Controllers}
$\qquad$\\ Consider the system \eqref{eq:Time-invariant Linear Discrete-time System}--\eqref{eq:Time-invariant control law} under the stated Assumptions~\ref{assumption:sub-Gaussian-TI}.
Given a finite controller space $\mathcal K=\{K_1,\dots,K_L\}$, generate the dataset $\mathcal{S} =\{X_{K_j}^{(i)}\}^{i=1,\cdots ,n}_{j=1,\cdots,L}$ according to the data-generating process. Let $P_0$ be any data-independent distribution over $\mathcal K$, let $\Omega\subset(0,+\infty)$ be a data-independent finite set and define
\begin{equation}\label{eq: gamma}
\Gamma=\Omega\cap\Bigl(0,\ \operatorname*{ess\,inf}_{K\in\mathcal K}\tfrac{1}{4\sigma_w^2\,\rho_Z(K)\,\rho_M(K)^{2T}}\Bigr),
\end{equation}
where $\rho_Z(K):=\|Q\|_F+\|K\|^2\|R\|_F$ and $\rho_M(K):=\rho_A+\rho_B\|K\|$, with $ \rho_A = \max \{|a_1|,|a_2|\}$, $\rho_B = \max \{|b_1|,|b_2|\}$, and $a_1,a_2,b_1,b_2$ are the constants defined in Assumptions~\ref{assumption:sub-Gaussian-TI}.
Then, $\forall \delta\in(0,1)$,
\begin{align}
&\mathbb{P}\!\left(\forall P,
\forall \lambda\in\Gamma,\;\bar
C(P)\le \tilde C(P)+\frac{\lambda B_{\text{cost}}^{2}}{8n}
+ \frac{KL(P\Vert P_0)+\ln \frac{card(\Gamma)}{\delta}}{\lambda}
\right)\notag \\ 
&\ge 1-\delta,
\label{eq: PAC-Bayes Bound for Finite Controllers}
\end{align}
where $B_{cost}$ is defined by Eq.~\eqref{eq:Bcost}.
\end{theorem}

\begin{proof}
As discussed in Section~\ref{sec:problem}, the quadratic cost can be expressed as a function of $(K,A,B,W)$. 
In particular, the Gibbs expected quadratic cost satisfies
\[
\mathbb{E}_{K}\mathbb{E}_{X_K}\bigl[C_q(K,X_K)\bigr]
= \mathbb{E}_{K}\mathbb{E}_{A,B,W}\bigl[C_q\bigl(K, A,B,W\bigr)\bigr].
\]
First, the generalization gap under $P$ is defined as
\[
G:=\bar C(P)-\tilde C(P)=\sum_{j=1}^n P_j\big(C(K_j)-\hat C(K_j)\big).
\]
We follow the derivation steps in~\citep{c3}, where the Kullback–Leibler change of measure was introduced, to upper bound the generalization gap between the Gibbs expected and Gibbs empirical cost. For any $\lambda > 0$, it holds that:
\begin{equation}\label{eq:thm-1}
G\;\le\;\frac{1}{\lambda}\left\{
KL(P\Vert P_0)\;+\;\ln\!\left(\sum_{j=1}^n P_{0j}e^{\lambda\,[C(K_j)-\hat C(K_j)]}\right)
\right\}.
\end{equation}
In the subsequent proof, we write $\hat C(K_j)$ to emphasize that it is the empirical cost computed from the realizations $A$, $B$ and $W$. Consider that $A'$, $B'$ and $W'$ are the independent copies of the realizations, then define $\hat C'(K_j)$ as the empirical cost calculated by using the realizations $A'$, $B'$ and $W'$.
As in~\citep{c3}, we now focus on the term
\[
f(A,B,W):=\sum_{j=1}^n P_{0j}e^{\lambda\,[C(K_j)-\hat C(K_j)]}.
\]
By applying Markov’s inequality, for any $\delta\in(0,1)$, we obtain
\begin{equation*}\label{eq:markov}
\mathbb{P}\!\left(f(A,B,W)\le \frac{1}{\delta}\,\mathbb E_{A',B',W'}f(A',B',W')\right)\;\ge\;1-\delta.
\end{equation*}
Because $P_0$ is data-independent, we have
\begin{align*}
\mathbb E_{A',B',W'}f(A',B',W')
&=\mathbb E_{A',B',W'}\sum_{j=1}^nP_{0j}e^{\lambda[C(K_j)-\hat C'(K_j)]}\notag\\
&=\sum_{j=1}^n P_{0j}\,\mathbb E_{A',B',W'}e^{\lambda[C(K_j)-\hat C'(K_j)]}.
\label{eq:exchange}
\end{align*}
By the same argument in~\citep{c3}, we obtain that with probability at least $1-\delta$, Eq.~\eqref{eq:thm-1} is less than
\begin{equation}\label{eq:thm-2}
 \frac{1}{\lambda}\left\{KL(P\Vert P_0)+
\ln\!\left(\frac{1}{\delta}\sum_{j=1}^n P_{0j}\,\mathbb E_{A',B',W'}e^{\lambda[C(K_j)-\hat C'(K_j)]}\right)\right\}.
\end{equation}
The details of the above steps can be found in~\citep{c3}.
We can use the sub-Gaussian property to show that inequality  $\mathbb E_{A',B',W'} [e^{\lambda (C(K_j) - \hat{C}(K_j))}] \le  e^{\frac{\lambda^2 B_{cost}^2 (K_j)}{8m}}$ holds for $|\lambda| \le \tfrac{1}{4\sigma_w^2\,\rho_Z(K_j)\,\rho_M(K_j)^{2T}}$ (a detailed proof is given later). Then this further implies that Eq.\eqref{eq:thm-2} is bounded by
  \begin{align}\label{eq:thm-3}
& \frac{1}{\lambda}\left\{KL(P\Vert P_0)+
\ln\!\left(\frac{1}{\delta}\sum_{j=1}^n P_{0j}\,\exp\!\left(\frac{\lambda^2}{8m}\,B_{\text{cost}}^2(K_j)\right)\right)\right\}\notag\\
\le\ &\frac{1}{\lambda}\left\{KL(P\Vert P_0)+
\ln\!\left(\frac{1}{\delta}\,\exp\!\Big(\tfrac{\lambda^2}{8m}\,B_{\text{cost}}^{2}\Big)\right)\right\}.
\end{align}
   Based on Theorem 2.4 in~\citep{c1}, we can show that the following quantity is an upper bound on ~\eqref{eq:thm-3}, which holds uniformly for all $\lambda$ in $\Gamma$, defined in~\eqref{eq:thm-4}.
    \begin{align}\label{eq:thm-4}
 & \frac{1}{\lambda} [KL(P \parallel P_0) + ln \frac{card(\Gamma)}{\delta} + \frac{\lambda^2 B_{cost}^2}{8m}] \\
    &= \frac{\lambda B_{cost}^2}{8m} + \frac{KL(P \parallel P_0) + \ln \frac{card(\Gamma)}{\delta}}{\lambda}.\notag
\end{align}
Thus, the inequality in Eq.~\eqref{eq: PAC-Bayes Bound for Finite Controllers} holds. 

We now present the main ideas for proving the inequality
\begin{equation}\label{eq: MGF Bcost}
    \mathbb{E}_{A',B',W'} \left[e^{\lambda (C(K) - \hat{C}(K))} \right] \le  e^{\frac{\lambda^2 B_{\text{cost}}^2}{8m}},
\end{equation}
for $|\lambda| \le \operatorname*{ess\,inf}_{K \in \mathcal{K}} \frac{1}{4 \sigma_w^2 \cdot \rho_Z(K) \cdot \rho_M(K)^{2T}} $.
Consider the system \eqref{eq:Time-invariant Linear Discrete-time System}--\eqref{eq:Time-invariant control law} under the stated Assumptions~\ref{assumption:sub-Gaussian-TI}.
For notational simplicity, we write $\mathbb{E}$ in place of $\mathbb{E}_{A',B',W'}$ throughout the following proof.
For $t=0,\dots,T-1$, we define $M := A + B K$, so that
\begin{equation*}
x(t+1) = M x(t) + w(t).
\end{equation*}
For a fixed realization of $(A,B,W)$ (equivalently, of the trajectory $X_K$), the quadratic cost is
\begin{equation*}
C_\text{q}(K,X_K) = \sum_{t=0}^{T-1} x(t)^\top Z x(t) + x(T)^\top Q x(T),
\end{equation*}
where
\[
Z := Q + K^\top R K.
\]

We first expand $x(t)$ in terms of the noise variables.
From the recursion $x(t+1) = M x(t) + w(t)$ and $x(0)=0$, we obtain
\begin{equation*}
x(t) = \sum_{k=0}^{t-1} M^{\,t-1-k} w(k), \qquad t=1,\dots,T.
\end{equation*}
Substituting this into the cost expression and grouping terms by the noise indices yields
\begin{equation*}
C_\text{q}(K,X_K)  = \sum_{t=0}^T \sum_{i,j=0}^{t-1} w(i)^\top D_{t,ij} w(j),
\end{equation*}
where
\begin{equation*}
D_{t,ij} := M^{\,t-1-i\,\top} Z M^{\,t-1-j}.
\end{equation*}
By grouping terms over time, we define
\begin{equation*}
H_i = \sum_{t=i+1}^T D_{t,ii}, \quad
H_{ij} = \sum_{t=\max(i,j)+1}^T D_{t,ij},\quad i\neq j,
\end{equation*}
and rewrite the cost as
\begin{align*}\label{eq:Cq-H-expansion}
C_\text{q}(K,X_K) &= \sum_{i=0}^{T-1} w(i)^\top H_i w(i)\\
&+ \sum_{0 \le i < j \le T-1} \bigl(w(i)^\top H_{ij} w(j) + w(j)^\top H_{ij}^\top w(i)\bigr).
\end{align*}

Using the spectral decomposition $H_i = U \Lambda U^\top$ and defining $\tilde{w}(i) := U^\top w(i)$, we have
\begin{equation*}
w(i)^\top H_i w(i) = \tilde{w}(i)^\top \Lambda \tilde{w}(i),
\end{equation*}
where $\Lambda = \mathrm{diag}(\mu_1, \dots, \mu_{d_x})$ is a diagonal matrix of eigenvalues, where $\mu_i$ is the $i^{th}$ eigenvalue of $H_i$. Since $\Lambda$ is diagonal and positive semidefinite, we can upper bound the quadratic form as
\begin{equation*}
\tilde{w}(i)^\top \Lambda \tilde{w}(i) = \sum_{j=1}^{d_x} \mu_j \tilde{w}^2_j(i) \le \|\Lambda\| \cdot \sum_{j=1}^{d_x} \tilde{w}^2_j(i) = \|\Lambda\| \cdot \|\tilde{w}(i)\|_2^2,
\end{equation*}
due to $\|\Lambda\| = \max_j \mu_j$. The inequality follows from the fact that each $\mu_j \le \|\Lambda\|$.

Since the noise vector $w(i) $ has independent $\sigma_w^2$-sub-Gaussian entries, and $U \in \mathbb{R}^{d_x \times d_x}$ is an orthogonal matrix, for any $v \in \mathbb{R}^{d_x }$ with $\|v\|_2 = 1$,  by Def.~\ref{def:sub-Gaussian} and Def.~\ref{def:sub-Gaussian vector}, we have
\begin{equation*}
\mathbb{E} \left[ e^{\lambda v^\top \tilde{w}(i)} \right]
= \mathbb{E} \left[ e^{\lambda (Uv)^\top w(i)} \right]
\le \exp\left( \frac{\lambda^2 \sigma_w^2}{2} \right),
\end{equation*}
where we used that $Uv$ is a unit vector and that linear combinations of independent sub-Gaussian variables remain sub-Gaussian. Hence, $\tilde{w}(i)$ is a $\sigma_w^2$-sub-Gaussian vector.

We note that the definition of a sub-Gaussian vector guarantees sub-Gaussian concentration along any unit direction $v \in \mathbb{R}^{d_x}$. In particular, taking $v = e_j$, the $j$-th standard basis vector, selects the $j$-th coordinate of $\tilde{w}(i)$, i.e., $\tilde{w}_j(i) = e_j^\top \tilde{w}(i)$. Since $\|e_j\|_2 = 1$, by Def.~\ref{def:sub-Gaussian}, applying the inequality above with $v = e_j$ gives:
\begin{equation*}
\mathbb{E} \left[ e^{\lambda \tilde{w}_j(i)} \right] = \mathbb{E} \left[ e^{\lambda e_j^\top \tilde{w}(i)} \right] \le \exp\left( \frac{\lambda^2 \sigma_w^2}{2} \right),
\end{equation*}
which shows that each entry $\tilde{w}_j(i)$ is a $\sigma_w^2$-sub-Gaussian random variable.

Since each $\tilde{w}_j(i)$ is a sub-Gaussian variable with parameter $\sigma_w^2$, we apply the result from \citep[Appendix B]{honorio}, which gives:
\begin{equation*}
\mathbb{E}\left[e^{\lambda (\tilde{w}_j(i)^2 - \mathbb{E}[\tilde{w}_j(i)^2])} \right]
\le \exp(16 \lambda^2 \sigma_w^4).
\end{equation*}
Then we have
\begin{equation*}
\tilde{w}(i)^\top \Lambda \tilde{w}(i) = \sum_{j=1}^{d_x} \mu_j \tilde{w}_j(i)^2,
\end{equation*}
and we can upper bound the centered sum as
\begin{equation*}
\sum_{j=1}^{d_x} \mu_j (\tilde{w}_j(i)^2 - \mathbb{E}[\tilde{w}_j(i)^2])  
\le \|\Lambda\| \cdot \sum_{j=1}^{d_x} (\tilde{w}_j(i)^2 - \mathbb{E}[\tilde{w}_j(i)^2]),
\end{equation*}due to $\|\Lambda\| = \max_j|\mu_j|$.
Applying the bound to each $\tilde{w}_j(i)$ and combining them, we obtain
\begin{align*}
&\mathbb{E} \left[ e^{\lambda (\tilde{w}(i)^\top \Lambda \tilde{w}(i) - \mathbb{E}[\tilde{w}(i)^\top \Lambda \tilde{w}(i)]) } \right] = \mathbb{E} \left[ e^{\lambda ( w(i)^\top H_i w(i) - \mathbb{E}[w(i)^\top H_i w(i)]) } \right]  \\
&\le \exp\left( 16 \lambda^2 \|\Lambda\|^2 \sigma_w^4 d_x \right),
\end{align*}
for all $\lambda$ such that $|\lambda| \le \frac{1}{4 \sigma_w^2 \|\Lambda\|}$, by the result from \citep[Appendix B]{honorio}.
Hence, we have
\begin{align*}
\mathbb{E} \left[ e^{\lambda ( w(i)^\top H_i w(i) - \mathbb{E}[w(i)^\top H_i w(i)]) } \right]
\le \exp\left( 16 \lambda^2 \|H_i\|^2 \sigma_w^4 d_x \right),
\end{align*}
for all $\lambda$ such that $|\lambda| \le \frac{1}{4 \sigma_w^2 \|H_i\|}$.

We now provide a detailed derivation for bounding the moment generating function (MGF) of the cross term $w(i)^\top H_{ij} w(j)$, where $w(i)$ and $w(j)$ are independent random vectors with sub-Gaussian entries and $i \neq j$. Our goal is to upper bound the centered MGF:
\begin{equation*}
\mathbb{E} \left[ e^{\lambda ( w(i)^\top H_{ij} w(j) - \mathbb{E}[w(i)^\top H_{ij} w(j)]) } \right].
\end{equation*}

We show how to express the bilinear form $w(i)^\top H_{ij} w(j)$ as a quadratic form, which allows us to apply the result from \citep[Appendix B]{honorio}.

Let $\zeta_{ij} \in \mathbb{R}^{2d_x}$ be the concatenated vector
\begin{equation*}
\zeta_{ij} := \begin{bmatrix} w(i) \\ w(j) \end{bmatrix},
\end{equation*}
and define the symmetric block matrix $\Psi_{ij} \in \mathbb{R}^{2d_x \times 2d_x}$ as
\begin{equation*}
\Psi_{ij} := \begin{bmatrix}
0 & H_{ij} \\
H_{ij}^\top & 0
\end{bmatrix}.
\end{equation*}
We compute the quadratic form $\zeta_{ij}^\top \Psi_{ij} \zeta_{ij}$:
\begin{align*}
\zeta_{ij}^\top \Psi_{ij} \zeta_{ij} 
&= 
\begin{bmatrix} w(i)^\top & w(j)^\top \end{bmatrix}
\begin{bmatrix}
0 & H_{ij} \\
H_{ij}^\top & 0
\end{bmatrix}
\begin{bmatrix} w(i) \\ w(j) \end{bmatrix} \\
&= w(i)^\top H_{ij} w(j) + w(j)^\top H_{ij}^\top w(i) \\
&= 2 w(i)^\top H_{ij} w(j),
\end{align*}
where the last step uses the fact that both terms are scalars and equal.

Thus, we have:
\begin{equation*}
w(i)^\top H_{ij} w(j) = \frac{1}{2} \zeta_{ij}^\top \Psi_{ij} \zeta_{ij}.
\end{equation*}
This shows that the centered bilinear form can be written as a centered quadratic form over the augmented vector $\zeta_{ij}$.
\begin{equation*}
w(i)^\top H_{ij} w(j) - \mathbb{E}[w(i)^\top H_{ij} w(j)] = \frac{1}{2} \left( \zeta_{ij}^\top \Psi_{ij} \zeta_{ij} - \mathbb{E}[\zeta_{ij}^\top \Psi_{ij} \zeta_{ij}] \right).
\end{equation*}

Then we have
\begin{equation*}
\mathbb{E} \left[ e^{\lambda (w(i)^\top H_{ij} w(j) - \mathbb{E}[w(i)^\top H_{ij} w(j)])} \right]
= \mathbb{E} \left[ e^{\frac{\lambda}{2} (\zeta_{ij}^\top \Psi_{ij} \zeta_{ij} - \mathbb{E}[\zeta_{ij}^\top \Psi_{ij} \zeta_{ij}])} \right].
\end{equation*}

We now perform a spectral decomposition of the symmetric matrix $\Psi_{ij}$: let
\begin{equation*}
\Psi_{ij} = \mathcal{U}_{ij} \Sigma_{ij} \mathcal{U}_{ij}^\top,
\end{equation*}
where $\Sigma_{ij} = \mathrm{diag}(\nu_1, \dots, \nu_{2d_x})$ is a diagonal matrix of eigenvalues and $\mathcal{U}_{ij} \in \mathbb{R}^{2d_x \times 2d_x}$ is orthogonal. Define the rotated noise vector
\begin{equation*}
\tilde{\zeta}_{ij} := \mathcal{U}_{ij}^\top \zeta_{ij} \in \mathbb{R}^{2d_x}.
\end{equation*}
Then we can rewrite the quadratic form as
\begin{equation*}
\zeta_{ij}^\top \Psi_{ij} \zeta_{ij} = \tilde{\zeta}_{ij}^\top \Sigma_{ij} \tilde{\zeta}_{ij} = \sum_{k=1}^{2d_x} \nu_k \tilde{\zeta}_{ij,k}^2.
\end{equation*}

Since $w(i)$ and $w(j)$ are independent and each has i.i.d. sub-Gaussian entries with parameter $\sigma_w^2$, the concatenated vector $\zeta_{ij}$ is also sub-Gaussian with parameter $\sigma_w^2$. Moreover, similar to the analysis of $\tilde{w}(i)$, since $\mathcal{U}_{ij}$ is orthogonal, $\tilde{\zeta}_{ij} = \mathcal{U}_{ij}^\top \zeta_{ij}$ is sub-Gaussian vector with parameter $\sigma_w^2$.

By the Def.~\ref{def:sub-Gaussian vector} for any $v \in \mathbb{R}^{2d_x}$ with $\|v\|_2 = 1$, we have
\begin{equation*}
\mathbb{E} \left[ e^{\lambda v^\top \tilde{\zeta}_{ij}} \right]
= \mathbb{E} \left[ e^{\lambda (\mathcal{U}_{ij} v)^\top \zeta_{ij}} \right]
\le \exp\left( \frac{\lambda^2 \sigma_w^2}{2} \right).
\end{equation*}
Hence, $\tilde{\zeta}_{ij}$ is a sub-Gaussian vector with parameter $\sigma_w^2$.

In particular, for each $k$, taking $v = e_k$ (the $k$-th standard basis vector in $\mathbb{R}^{2d_x}$), we obtain
\begin{equation*}
\mathbb{E} \left[ e^{\lambda \tilde{\zeta}_{ij,k}} \right]
= \mathbb{E} \left[ e^{\lambda e_k^\top \tilde{\zeta}_{ij}} \right]
\le \exp\left( \frac{\lambda^2 \sigma_w^2}{2} \right),
\end{equation*}
which shows that each coordinate $\tilde{\zeta}_{ij,k}$ is a sub-Gaussian random variable with parameter $\sigma_w^2$.

Since each $\tilde{\zeta}_{ij,k}$ is sub-Gaussian with parameter $\sigma_w^2$, we apply the scalar MGF bound from \citep[Appendix B]{honorio}, which gives:
\begin{equation*}
\mathbb{E} \left[ e^{\lambda (\tilde{\zeta}_{ij,k}^2 - \mathbb{E}[\tilde{\zeta}_{ij,k}^2])} \right]
\le \exp(16 \lambda^2 \sigma_w^4).
\end{equation*}

Then, the centered quadratic form becomes:
\begin{equation*}
\tilde{\zeta}_{ij}^\top \Sigma_{ij} \tilde{\zeta}_{ij} - \mathbb{E}[\tilde{\zeta}_{ij}^\top \Sigma_{ij} \tilde{\zeta}_{ij}] = \sum_{k=1}^{2d_x} \nu_k (\tilde{\zeta}_{ij,k}^2 - \mathbb{E}[\tilde{\zeta}_{ij,k}^2]),
\end{equation*}
which can be bounded as
\begin{equation*}
\sum_{k=1}^{2d_x} \nu_k (\tilde{\zeta}_{ij,k}^2 - \mathbb{E}[\tilde{\zeta}_{ij,k}^2])
\le \|\Sigma_{ij}\| \cdot \sum_{k=1}^{2d_x} (\tilde{\zeta}_{ij,k}^2 - \mathbb{E}[\tilde{\zeta}_{ij,k}^2]),
\end{equation*}
where $\|\Sigma_{ij}\| := \max_k |\nu_k|$.

To bound the moment generating function of the centered bilinear form, we recall that
\begin{equation*}
\mathbb{E} \left[ e^{\lambda ( w(i)^\top H_{ij} w(j) - \mathbb{E}[w(i)^\top H_{ij} w(j)] )} \right]
= \mathbb{E} \left[ e^{\frac{\lambda}{2} ( \zeta_{ij}^\top \Psi_{ij} \zeta_{ij} - \mathbb{E}[\zeta_{ij}^\top \Psi_{ij} \zeta_{ij}] )} \right].
\end{equation*}
By the result from \citep[Appendix B]{honorio}, applying the composition of independent sub-Gaussian quadratic bounds and combining the $2d_x$ terms, we obtain:
\begin{equation*}
\mathbb{E} \left[ \exp\left( \lambda \sum_{k=1}^{2d_x} (\tilde{\zeta}_{ij,k}^2 - \mathbb{E}[\tilde{\zeta}_{ij,k}^2]) \right) \right]
\le \exp\left( 16 \cdot 2d_x \lambda^2 \sigma_w^4  \right).
\end{equation*}
By replacing $\lambda$ by $\frac{\lambda \|\Sigma_{ij}\|}{2}$ to account for the factor in front of the quadratic form, we conclude:
\begin{equation*}
\mathbb{E} \left[ e^{\frac{\lambda}{2} ( \zeta_{ij}^\top \Psi_{ij} \zeta_{ij} - \mathbb{E}[\zeta_{ij}^\top \Psi_{ij} \zeta_{ij}] )} \right]
\le \exp\left( 8 \lambda^2 \|\Sigma_{ij}\|^2 \sigma_w^4 d_x \right),
\end{equation*}
which implies
\begin{equation*}
\mathbb{E} \left[ e^{\lambda ( w(i)^\top H_{ij} w(j) - \mathbb{E}[w(i)^\top H_{ij} w(j)] )} \right]
\le \exp\left( 8 \lambda^2 \|\Sigma_{ij}\|^2 \sigma_w^4 d_x \right),
\end{equation*}
for all $\lambda$ such that $|\lambda| \le \frac{1}{4 \sigma_w^2 \|\Sigma_{ij}\|}$.
That is,
\begin{equation*}
\mathbb{E} \left[ e^{\lambda \left( w(i)^\top H_{ij} w(j) - \mathbb{E}[w(i)^\top H_{ij} w(j)] \right)} \right]
\le \exp\left( 8 \lambda^2 \|H_{ij}\|^2 \sigma_w^4 d_x \right),
\end{equation*}
for all $\lambda$ such that $|\lambda| \le \frac{1}{4 \sigma_w^2 \|H_{ij}\|}$.

We first consider the deviation of the quadratic cost for a single trajectory $A' \sim \mathcal D_A$, $B' \sim \mathcal D_B$, and $w'(t) \sim \mathcal D_w$. From the previous derivation using the structure of the cost in terms of $\{H_i, H_{ij}\}$ and the sub-Gaussianity of $w(t)$, we obtain:
\begin{align*}
&\mathbb{E}\left[e^{\lambda \left( \mathbb{E}[C_q(K, X_K')] - C_q(K, X_K') \right)} \right] \\
&= \mathbb{E} \left[ \exp\left( \lambda \sum_{i=0}^{T-1} \left( \mathbb{E}[w(i)^\top H_i w(i)] - w(i)^\top H_i w(i) \right) \right.\right. \\
& \left.\left.+ \sum_{0 \le i < j \le {T-1}} \left( \mathbb{E}[w(i)^\top H_{ij} w(j)] - w(i)^\top H_{ij} w(j) \right) \right) \right], \\
&\le  \mathbb{E} \left[ \exp\left( \lambda \sum_{i=0}^T \left( \mathbb{E}[w(i)^\top H_i w(i)] - w(i)^\top H_i w(i) \right) \right.\right. \\
& \left.\left.+ \sum_{0 \le i < j \le {T}} \left( \mathbb{E}[w(i)^\top H_{ij} w(j)] - w(i)^\top H_{ij} w(j) \right) \right) \right],
\end{align*}
for all $|\lambda| \le  \frac{1}{4 \sigma_w^2 \max\left\{ \max_i \|H_i\|, \max_{i < j} \|H_{ij}\| \right\}} $.
We apply the previous bounds for the diagonal and off-diagonal terms. By Jensen's inequality and the fact that the moment generating function of a sum is upper bounded by the product of the individual MGFs (for sub-exponential variables), we obtain:
\begin{align*}
&\mathbb{E}\left[e^{\lambda \left( \mathbb{E}[C_q(K, X_K')] - C_q(K, X_K') \right)} \right] \\
&\le \exp\left( \lambda^2 \sigma_w^4 d_x \left( \sum_{i=0}^T 16 \|H_i\|^2 + \sum_{0 \le i < j \le T} 8 \|H_{ij}\|^2 \right) \right),
\end{align*}
Then the deviation between the expected and empirical cost can be written as
\begin{equation*}
\mathbb{E} \left[ e^{\lambda (C(K) - \hat{C}(K))} \right]
= \mathbb{E} \left[ e^{ \lambda \left( \mathbb{E}[C_q(K,X_K)] - \frac{1}{m} \sum_{i=1}^m C_q(K, X^{(i)}_K) \right) }\right].
\end{equation*}

We now use linearity of expectation to rewrite:
\begin{align*}
&\mathbb{E}[C_q(K, X_K)] - \frac{1}{m} \sum_{i=1}^m C_q(K, X^{(i)}_K)\\
&= \frac{1}{m} \sum_{i=1}^m \left( \mathbb{E}[C_q(K, X_K)] - C_q(K, X^{(i)}_K) \right).
\end{align*}

Therefore,
\begin{align*}
&\mathbb{E} \left[ e^{\lambda (C(K) - \hat{C}(K))} \right] \\
&= \mathbb{E} \left[ e^{ \lambda \cdot \frac{1}{m} \sum_{i=1}^m \left( \mathbb{E}[C_q(K, X_K)] - C_q(K, X^{(i)}_K) \right) } \right] \\
&= \mathbb{E} \left[ \prod_{i=1}^m e^{ \frac{\lambda}{m} \left( \mathbb{E}[C_q(K, X_K)] - C_q(K, X^{(i)}_K) \right) } \right].
\end{align*}

Since the $C_q(K, X^{(i)}_K)$ are independent (due to the independent system parameters and noise sequence), each random variable $\mathbb{E}[C_q(K, X_K)] - C_q(K, X^{(i)}_K)$ is also independent. Hence, the product of exponentials factorizes:
\begin{align*}
&\mathbb{E} \left[ \prod_{i=1}^m e^{ \frac{\lambda}{m} \left( \mathbb{E}[C_q(K, X_K)] - C_q(K, X^{(i)}_K) \right) } \right] \\
&= \prod_{i=1}^m \mathbb{E} \left[ e^{ \frac{\lambda}{m} \left( \mathbb{E}[C_q(K, X_K)] - C_q(K, X^{(i)}_K) \right) } \right].
\end{align*}

Thus, we conclude:
\begin{equation*}
\mathbb{E} \left[ e^{\lambda (C(K) - \hat{C}(K))} \right]
= \prod_{i=1}^m \mathbb{E} \left[ e^{ \frac{\lambda}{m} \left( \mathbb{E}[C_q(K, X_K)] - C_q(K, X^{(i)}_K) \right) } \right].
\end{equation*}

From the previous derivation for a single trajectory, we know that for all $\lambda' \in \left[0, \frac{1}{4 \sigma_w^2 \max\left\{ \max_i \|H_i\|, \max_{i < j} \|H_{ij}\| \right\}} \right]$, we have:
\begin{align*}
&\mathbb{E} \left[ e^{ \lambda' \left( \mathbb{E}[C_q(K, X_K)] - C_q(K, X^{(i)}_K) \right) } \right] \\
&\le \exp\left( \lambda'^2 \sigma_w^4 d_x \cdot \left( \sum_{i=0}^T 16 \|H_i\|^2 + \sum_{0 \le i < j \le T}8\|H_{ij}\|^2 \right) \right).
\end{align*}

We apply this bound with $\lambda' = \lambda / m$, and note that all $m$ terms are identical. This gives:
\begin{align*}
&\mathbb{E} \left[ e^{\lambda (C(K) - \hat{C}(K))} \right]
= \prod_{i=1}^m \mathbb{E} \left[ e^{ \frac{\lambda}{m} \left( \mathbb{E}[C_q(K, X_K)] - C_q(K, X^{(i)}_K) \right) } \right]
\\
&\le \left( \exp\left( \frac{\lambda^2}{m^2} \cdot \sigma_w^4 d_x \cdot \left( \sum_{i=0}^T 16 \|H_i\|^2 + \sum_{0 \le i < j \le T} 8 \|H_{ij}\|^2 \right) \right) \right)^m \\
&= \exp\left( \frac{\lambda^2}{m} \cdot \sigma_w^4 d_x \cdot \left( \sum_{i=0}^T 16 \|H_i\|^2 + \sum_{0 \le i < j \le T} 8 \|H_{ij}\|^2 \right) \right).
\end{align*}

This implies that the quadratic cost is sub-exponential with variance proxy:
\begin{equation*}
B_{cost}^2(K) := \sigma_w^4 d_x \left( 128 \sum_{i=0}^T \|H_i\|^2 + 64 \sum_{0 \le i < j \le T} \|H_{ij}\|^2 \right).
\end{equation*}
We now derive an explicit upper bound on $B_{cost}^2(K)$ in terms of system-dependent quantities $\rho_Z(K)$ and $\rho_M(K)$. Recall the definitions:
\begin{equation*}
\rho_Z(K) := \|Q\|_F + \|K\|^2 \|R\|_F, \quad \rho_M(K) := \rho_A + \rho_B \|K\|,
\end{equation*}
where $\rho_A := \max\{|a_1|, |a_2|\}$ and $\rho_B := \max\{|b_1|, |b_2|\}$. These provide bounds on $\|Z\|$ and $\|M\|$:
\begin{equation*}
\|Z\| \le \rho_Z(K), \quad \|M\| \le \rho_M(K).
\end{equation*}

Each $H_i$ corresponds to the contribution of $w(i)$ to the quadratic cost:
\begin{equation*}
H_i = M^{T-i}{}^\top \cdots M^\top Z M \cdots M^{T-i}.
\end{equation*}
Hence, we can bound:
\begin{equation*}
\|H_i\| \le \|Z\| \cdot \|M\|^{2(T - i)} \le \rho_Z(K) \cdot \rho_M(K)^{2(T - i)}.
\end{equation*}
 
The cross-term matrix $H_{ij}$ arises from the interaction between $w(i)$ and $w(j)$:
\begin{equation*}
H_{ij} = M^{T - i}{}^\top \cdots M^\top Z M \cdots M^{T - j},
\end{equation*}
so that
\begin{align*}
&\|H_{ij}\| \le \|Z\| \cdot \|M\|^{T - i + T - j} = \|Z\| \cdot \|M\|^{2T - i - j}\\
&\le \rho_Z(K) \cdot \rho_M(K)^{2T - i - j}.
\end{align*}

Substituting into the expression:
\begin{align*}
B_{cost}^2(K)
&\le \sigma_w^4 d_x \cdot \rho_Z(K)^2 \cdot \left(
128 \sum_{i=0}^T \rho_M(K)^{4(T - i)} \right.\\
&\left.+
64 \sum_{0 \le i < j \le T} \rho_M(K)^{4T - 2(i + j)}
\right).
\end{align*}
Combining with Eq.~\eqref{eq:thm-3}, this gives a compact upper bound on $B_{\text{cost}}^2$ in terms of $\rho_Z(K)$, $\rho_M(K)$, $\sigma_w^2$, $d_x$, and $T$.

For the admissible range of $\lambda$, we need the maximum spectral norm across all $H_i$ and $H_{ij}$. Since 
\begin{equation*}
    \max_{i} \|H_i\| = \|H_0\| \le \rho_Z(K) \cdot \rho_M(K)^{2T},
\end{equation*}
and 
\begin{equation*}
    \max_{i<j} \|H_{ij}\| = \|H_{01}\| \le \rho_Z(K) \cdot \rho_M(K)^{2T-1},
\end{equation*}
we get
\begin{equation*}
|\lambda| \le  \operatorname*{ess\,inf}_{K \in \mathcal{K}}\frac{1}{4 \sigma_w^2 \cdot \rho_Z(K) \cdot \rho_M(K)^{2T}} .
\end{equation*}
Hence, we have obtained Eq.~\eqref{eq: MGF Bcost} and the expression for $B_{cost}$ Eq.~\eqref{eq:Bcost} used in Eq.~\eqref{eq: PAC-Bayes Bound for Finite Controllers}, which completes the proof.
\end{proof}

The bound in Eq.~\eqref{eq: PAC-Bayes Bound for Finite Controllers} certifies any data-dependent posterior $P$
by upper bounding its Gibbs expected cost with the Gibbs empirical cost
and a complexity term involving $B_{\rm cost}$, $n$, and
$KL(P\|P_0)$. Unlike classical bounded-cost PAC-Bayes bounds, Eq.~\eqref{eq: PAC-Bayes Bound for Finite Controllers} applies
directly to the quadratic cost through the MGF bound. The finite set $\Gamma$ allows data-dependent selection of $\lambda$ while preserving the PAC-Bayes guarantee.

\subsection{PAC-Bayes bound for an infinite controller space}
Unlike the finite controller case, where all controllers can be explicitly evaluated, the infinite controller case requires an empirical estimate over the controller space. The PAC--Bayes bound for the infinite controller space is stated below. 

\begin{theorem}[PAC-Bayes bound for infinite controller space]
\label{thm:pb-mc-hoeffding}
$\qquad$\\Consider the system \eqref{eq:Time-invariant Linear Discrete-time System}--\eqref{eq:Time-invariant control law} under the stated Assumptions~\ref{assumption:sub-Gaussian-TI}.
Let $\mathcal K\subseteq \mathbb R^{d_u\times d_x}$ be an infinite controller space. Let $P_0$ be any data-independent distribution over $\mathcal K$, and assume the controller is bounded: $\|K\|\le B_k$ for all $K\in\mathcal K$. For any probability distribution $P$ over controllers, define
\begin{equation}\label{eq: tilde CLprime}
    \tilde C _{L'}(P)\;:=\;\frac{1}{L'}\sum_{j=1}^{L'} \hat C(K_j),
\end{equation}
where $K_1,\dots,K_{L'}\overset{\mathrm{i.i.d.}}{\sim} P$, which are then used to generate the dataset $\mathcal{S} =\{X_{K_j}^{(i)}\}^{i=1,\cdots ,n}_{j=1,\cdots,L'}$ according to the data-generating process. Let $\Omega\subset(0,+\infty)$ be a finite data-independent set. Let $\Gamma$ and $B_{\text{cost}}$ be defined as in Eq.~\eqref{eq: gamma} and Eq.~\eqref{eq:Bcost}, respectively, and define
\begin{align*}
&\rho_Z^{\max} \;:=\; \|Q\|_F + B_k^2 \|R\|_F,
V(\mathcal S) \;:=\; \max_{j=1,\dots,L'}\frac{1}{n}\sum_{i=1}^n \sum_{t=1}^T \|x^{(i)}_{K_j}(t)\|^2,\\
&C_{\max}(\mathcal S) \;:=\; \rho_Z^{\max}\,V(\mathcal S).
\end{align*}
Then $\forall\delta\in(0,1)$ and $\forall\delta'\in(0,1)$ with $\delta+\delta'<1$, 
\begin{align}\label{eq:thm2-mc}
\mathbb{P}\!&\left( \forall P,
\forall \lambda\in\Gamma,\;\bar C(P)
\;\le\;
\tilde C_{L'}(P)
\;+\;
C_{\max}(\mathcal S)\,\sqrt{\frac{1}{2L'}\,\ln\!\frac{2}{\delta'}}\right.\notag\\
&\left. \;+\;\frac{\lambda\,B_{\text{cost}}^2}{8n}\;+\;
\frac{KL(P\Vert P_0)\;+\;\ln\!\frac{card(\Gamma)}{\delta}}{\lambda}\right)\ge 1-\delta-\delta'.
\end{align}
\end{theorem}
\begin{proof}
The proof has two components:
(i) a Monte Carlo deviation bound that controls the error incurred when approximating the (intractable) expectation over $K\sim P$ by $\tilde C_{L'}(P)$, and
(ii) a PAC--Bayes bound that controls the difference between $\bar C(P)$ and $\tilde C(P)$.
The PAC--Bayes part reuses the argument of Theorem~\ref{theorem: PAC-Bayes Bound for Controllers} and its proof; we only highlight the few notational changes needed in the infinite-controller setting.

Consider the dataset
\[
\mathcal S
=\bigl\{X_{K_j}^{(i)}\bigr\}_{i=1,\dots,n}^{j=1,\dots,L'}.
\]
For each controller $K_j$ and each trajectory index $i$, denote the corresponding state and input at time $t$ by
\[
x^{(i)}_{K_j}(t),\qquad
u^{(i)}_{K_j}(t) = K_j x^{(i)}_{K_j}(t).
\]
Using the Frobenius inner product and $\|xx^\top\|_F = \|x\|^2$,
\begin{align*}
\bigl(x^{(i)}_{K_j}(t)\bigr)^\top Q\,x^{(i)}_{K_j}(t)
&= \mathrm{tr}\!\Bigl(Q\,x^{(i)}_{K_j}(t)\bigl(x^{(i)}_{K_j}(t)\bigr)^\top\Bigr) \\
&\le \|Q\|_F\, \bigl\|x^{(i)}_{K_j}(t)\bigl(x^{(i)}_{K_j}(t)\bigr)^\top\bigr\|_F
= \|Q\|_F\,\bigl\|x^{(i)}_{K_j}(t)\bigr\|^2.
\end{align*}
For the control part,
\begin{align*}
\bigl(u^{(i)}_{K_j}(t)\bigr)^\top R\,u^{(i)}_{K_j}(t)
&= \bigl(x^{(i)}_{K_j}(t)\bigr)^\top K_j^\top R K_j\,x^{(i)}_{K_j}(t) \\
&\le \|K_j^\top R K_j\|\,\bigl\|x^{(i)}_{K_j}(t)\bigr\|^2 \\
&\le \|K_j\|^2 \,\|R\|\,\bigl\|x^{(i)}_{K_j}(t)\bigr\|^2 \\
&\le B_k^2\,\|R\|_F\,\bigl\|x^{(i)}_{K_j}(t)\bigr\|^2,
\end{align*}
where we used $\|K_j\|\le B_k$ and $\|R\| \le \|R\|_F$.
Combining the two pieces,
\begin{align*}
C_q\!\bigl(K_j,X_{K_j}^{(i)}\bigr)
&\le \sum_{t=0}^T
\Bigl[\bigl(x^{(i)}_{K_j}(t)\bigr)^\top Q\,x^{(i)}_{K_j}(t)
+\bigl(u^{(i)}_{K_j}(t)\bigr)^\top R\,u^{(i)}_{K_j}(t)\Bigr] \\
&\le \bigl(\|Q\|_F + B_k^2\|R\|_F\bigr)
\sum_{t=1}^T \bigl\|x^{(i)}_{K_j}(t)\bigr\|^2.
\end{align*}
Averaging over $i=1,\dots,n$, the empirical cost of $K_j$ satisfies
\begin{align*}
\hat C(K_j)
&:= \frac{1}{n}\sum_{i=1}^n C_q\!\bigl(K_j,X_{K_j}^{(i)}\bigr) \\
&\le
\bigl(\|Q\|_F + B_k^2\|R\|_F\bigr)
\frac{1}{n}\sum_{i=1}^n\sum_{t=1}^T \bigl\|x^{(i)}_{K_j}(t)\bigr\|^2.
\end{align*}
By definition of $\rho_Z^{\max}$ and $V(\mathcal S)$,
\begin{equation}
0 \;\le\; \hat C(K_j)
\;\le\; \rho_Z^{\max}\,
\frac{1}{n}\sum_{i=1}^n\sum_{t=1}^T \bigl\|x^{(i)}_{K_j}(t)\bigr\|^2
\;\le\; C_{\max}(\mathcal S),
\label{eq:range-mc}
\end{equation}
for all controllers with $\|K_j\|\le B_k$, given the dataset $\mathcal S$.

We now quantify the error introduced by approximating the expectation over $K\sim P$ using finitely many samples $K_1,\dots,K_{L'}$.
Define
\begin{equation*}
Y_j := \hat C(K_j),\qquad j=1,\dots,L'.
\end{equation*}
Conditional on the trajectories $\mathcal S$, the random variables $Y_1,\dots,Y_{L'}$ are i.i.d.\ (because the $K_j$ are i.i.d.\ from $P$) and, by \eqref{eq:range-mc}, they are bounded as
\[
0 \;\le\; Y_j \;\le\; C_{\max}(\mathcal S),\qquad j=1,\dots,L'.
\]
By construction,
\begin{equation*}
\mathbb E\bigl[Y_j \,\big|\, \mathcal S\bigr]
= \mathbb E_{K_j\sim P}\bigl[\hat C(K_j)\,\big|\,\mathcal S\bigr]
= \tilde C(P),
\end{equation*}
since $\mathcal S$ is independent of the subsequent draw of $K_j$ from $P$.

Define the Monte Carlo estimate $\tilde C_{L'}(P)$ as in \eqref{eq: tilde CLprime},
\begin{equation*}
\tilde C_{L'}(P)
= \frac{1}{L'}\sum_{j=1}^{L'} Y_j.
\end{equation*}
Hoeffding's inequality for i.i.d.\ random variables in $[0,C_{\max}(\mathcal S)]$ yields, for any $t>0$,
\begin{equation*}
\mathbb P\!\left(
\tilde C(P) - \tilde C_{L'}(P) \ge t
\,\Big|\, \mathcal S
\right)
\le
\exp\!\left(
-\frac{2L' t^2}{C_{\max}(\mathcal S)^2}
\right).
\end{equation*}
Set the right-hand side equal to $\delta'/2$ (one-sided Hoeffding bound) and solve for $t$:
\[
t
=
C_{\max}(\mathcal S)\,
\sqrt{\frac{1}{2L'}\ln\!\frac{2}{\delta'}}.
\]
Therefore, for any $\delta'\in(0,1)$,
\begin{equation}
\mathbb P\!\left(
\tilde C(P)
\le
\tilde C_{L'}(P)
+
C_{\max}(\mathcal S)\,
\sqrt{\frac{1}{2L'}\ln\!\frac{2}{\delta'}}
\right)
\;\ge\; 1-\delta',
\label{eq:hoeffding-mc}
\end{equation}
where the probability is over the joint randomness of $K_1,\dots,K_{L'}$ and the trajectories in $\mathcal S$.

We replace the definition of $G$ in the proof of Theorem~\ref{theorem: PAC-Bayes Bound for Controllers} with the following expression.
\begin{align*}
G
:= \bar C(P)-\tilde C(P)
&= \int C(K_j)\,{\rm d}P(K_j)\;-\;\int \hat C(K_j)\,{\rm d}P(K_j) \\
&= \int \big[C(K_j)-\hat C(K_j)\big]\,{\rm d}P(K_j).
\end{align*}
By following the same sequence of derivation steps as in the proof of Theorem~\ref{theorem: PAC-Bayes Bound for Controllers}, we
obtain the PAC--Bayes bound for the infinite controller space.
The resulting bound has the same form as that in Eq.~\eqref{eq: PAC-Bayes Bound for Finite Controllers}, since the finite controller
space can be viewed as a special case of the infinite controller setting.

Define the two events
\begin{align*}
E := \Bigl\{&\forall P,\;\forall \lambda\in\Gamma:\;
\bar C(P)
\le
\tilde C(P)
+
\frac{\lambda B_{\text{cost}}^2}{8n}
\\ 
&+\frac{KL(P\Vert P_0) + \ln \frac{\mathrm{card}(\Gamma)}{\delta}}{\lambda}
\Bigr\},\\[0.3em]
F := \Bigl\{&
\tilde C(P)
\le
\tilde C_{L'}(P)
+
C_{\max}(\mathcal S)\,
\sqrt{\tfrac{1}{2L'}\ln \tfrac{2}{\delta'}}
\Bigr\}.
\end{align*}
We have $\mathbb P(E)\ge 1-\delta$, and  $\mathbb P(F)\ge 1-\delta'$ according to Eq.~\eqref{eq:hoeffding-mc}.

On the intersection $E\cap F$ we have, simultaneously for all posteriors $P$ and all $\lambda\in\Gamma$,
\begin{align*}
\bar C(P)
&\le
\tilde C(P)
+
\frac{\lambda B_{\text{cost}}^2}{8n}
+
\frac{KL(P\Vert P_0) + \ln \frac{\mathrm{card}(\Gamma)}{\delta}}{\lambda} \\
&
\;\le\;
\underbrace{\tilde C_{L'}(P)}_{\mathrm{MC\ estimate\ over\ }K}
\;+\;
\underbrace{C_{\max}(\mathcal S)\,\sqrt{\frac{1}{2L'}\,\ln\!\frac{2}{\delta'}}}_{\mathrm{MC\ deviation\ (Hoeffding)}}
+
\frac{\lambda B_{\text{cost}}^2}{8n}
\\&+
\frac{KL(P\Vert P_0) + \ln \frac{\mathrm{card}(\Gamma)}{\delta}}{\lambda},
\end{align*}
which is exactly the claimed inequality~\eqref{eq:thm2-mc}.

Finally, by the union bound,
\begin{align*}
\mathbb P(E\cap F)
&\;\ge\;
1-\mathbb P(E^c\cup F^c)
\;\ge\;
1-\bigl(\mathbb P(E^c)+\mathbb P(F^c)\bigr)\\&
\;\ge\;
1-(\delta+\delta').
\end{align*}
Thus~\eqref{eq:thm2-mc} holds with probability at least $1-\delta-\delta'$,
uniformly over all posteriors $P$ and all $\lambda\in\Gamma$.
\end{proof}

Compared to Eq.~\eqref{eq: PAC-Bayes Bound for Finite Controllers}, the infinite-space bound replaces
$\tilde C(P)$ by the Monte Carlo estimate $\tilde C_{L'}(P)$.
The additional Hoeffding term controls this approximation error
and decreases as $1/\sqrt{L'}$.

\section{Learning Controllers by Minimizing PAC-Bayes Bounds}
\label{sec:algorithm}
In this section, we discuss a novel learning algorithm as our Objective 2. In light of Theorem~\ref{theorem: PAC-Bayes Bound for Controllers} and Theorem~\ref{thm:pb-mc-hoeffding}, a natural choice is to find a data-dependent $P$ (posterior distribution) by minimizing the right-hand side of the PAC--Bayes bounds in Eq.~\eqref{eq: PAC-Bayes Bound for Finite Controllers} or~\eqref{eq:thm2-mc}. It turns out that by minimizing a similar expression, but replacing the term $B_{cost}$ with an empirical version (details below), yields better performance. We want to emphasize that the bounds in Theorem~\ref{theorem: PAC-Bayes Bound for Controllers} and~\ref{thm:pb-mc-hoeffding} apply to any learned distribution $P$, and hence also apply to our learned distribution. 

The quantity $\hat B_{cost}$ is defined based on the training trajectories $\mathcal{S} = \{X_{K_j}^{(i)}\}^{i=1,\dots,n}_{j=1,\dots,L}$ as follows. In the proof, we show that for each controller $K$, the analysis establishes a moment generating function inequality shown in Eq.~\eqref{eq: MGF Bcost},
where the expectation is taken with respect to the randomness in $(A,B,W)$ that generates the $n$ training trajectories under $K$, and $C(K)$ denotes the corresponding expected cost. Writing
$Z_K = C(K)-\hat C(K)$ and expanding the log-moment generating function
at $\lambda=0$ gives
\[
\log \mathbb{E}_{A,B,W} e^{\lambda Z_K}
= \frac{\lambda^2}{2}Var(Z_K) + o(\lambda^2)
= \frac{\lambda^2}{2}Var(\hat C(K)) + o(\lambda^2).
\]
Let $X_{K,i}$ denote the $i$-th simulated trajectory under controller $K$, and
define the per-trajectory quadratic cost
\[
C_i(K) := C_q\bigl(K, X^{(i)}_K\bigr), \quad i = 1,\dots,n.
\]
Then $\hat C(K)$ is the average of the $n$ i.i.d.\ trajectory costs
$C_1(K),\dots,C_n(K)$, so that we have
$
Var(\hat C(K)) = \frac{1}{n}Var\bigl(C_q(K,X_K)\bigr).
$
Hence, for small $\lambda$, the left-hand side of \eqref{eq: MGF Bcost} behaves like
\[
\log \mathbb{E}_{A,B,W} e^{\lambda( C(K)-\hat C(K))}
\approx \frac{\lambda^2}{2n}Var\bigl(C_q(K,X_K)\bigr).
\]

Comparing the leading $\lambda^2$ terms on both sides of
\eqref{eq: MGF Bcost} suggests that $B_{cost}(K)^2$ should be of
the order of $4 Var(C_q(K,X_K))$, i.e., suggesting
\[
B_{cost}(K) \approx 2\sqrt{Var(C_q(K,X_K))}.
\]

In practice, $Var(C_q(K,X_K))$ is unknown. For each controller $K_j$, we use the $n$ trajectories $\{X_{K_j}^{(i)}\}_{i=1}^n$ in $\mathcal{S}$ and their corresponding costs $C_q\bigl(K_j, X^{(1)}_{K_j}\bigr),\dots,C_q\bigl(K_j, X^{(n)}_{K_j}\bigr)$ to estimate the variance via
\[
\widehat{Var}(C_q(K_j,X_{K_j}))
=
\frac{1}{n-1}\sum_{i=1}^n (C_q(K_j,X^{(i)}_{K_j}) - \hat C(K_j))^2,
\]
and define the empirical 
\[
\hat B_{cost}(K_j)
=
2 c_B \sqrt{\widehat{Var}(C_q(K_j,X_{K_j}))},
\]
where $c_B \ge 1$ is a safety factor. To obtain a single global constant,
we set
\begin{equation}\label{eq: empircal Bcost}
    \hat B_{cost}
=
\max_{1\le j\le L} \hat B_{cost}(K_j).
\end{equation}
This data-driven $\hat B_{cost}$ is then used in the
training objective as a numerically reasonable approximation of the theoretical constant appearing in \eqref{eq: MGF Bcost}.

\subsection{Finite controller space: PAC-Bayes controller learning algorithm}

We now consider Case~1, $\mathcal{K}=\{K_1,\dots,K_L\}$. For the controller set $\mathcal{K}$, a trajectory dataset $\mathcal{S}=\{X_{K_j}^{(i)}\}^{i=1,\dots,n}_{j=1,\dots,L}$ can be generated using the data-generating process described in Section~\ref{sec:thm}. The training objective is the right-hand side of inequality~\eqref{eq: PAC-Bayes Bound for Finite Controllers} with $B_{cost}$ replaced by $\hat B_{cost}$. The specific learning procedure is shown in Algorithm~\ref{alg:PAC-Bayes controller learning for time-invariant linear discrete-time system-finite}.
\begin{algorithm}
    \caption{PAC-Bayes controller learning for finite controller space}\label{alg:PAC-Bayes controller learning for time-invariant linear discrete-time system-finite}
    \begin{algorithmic}[1]
        \State \textbf{Inputs:}
        \State $\mathcal{S} =\{X_{K_j}^{(i)}\}^{i=1,\cdots ,n}_{j=1,\cdots,L}$: trajectory dataset generated according to the data-generating process described in Section~\ref{sec:thm}, with $L$ controllers and $n$ samples per controller
 
        \State $\delta\in(0,1)$: confidence parameter 
        \State $P_0$: prior distribution over controller space $\mathcal K$     
        \State $\Gamma$: a finite set of candidate values for the parameter $\lambda$ 
        \State Compute empirical $\hat B_{cost}$ according to Eq.~\eqref{eq: empircal Bcost}
        \State \textbf{Learning stage:}
\For{$\lambda \in \Gamma$}
  \[P_\lambda \in \arg\min_{P}
    \Bigl[
      \tilde C(P)
      + \frac{\lambda \hat B_{cost}^{2}}{8n}
      + \frac{KL(P\Vert P_0)+\ln \frac{card(\Gamma)}{\delta}}{\lambda}
    \Bigr].
  \]
\EndFor
\[
  \lambda^\star \in \arg\min_{\lambda \in \Gamma}
  \Bigl[
    \tilde C(P_\lambda)
    + \frac{\lambda \hat B_{cost}^{2}}{8n}
    + \frac{KL(P_\lambda\Vert P_0)+\ln \frac{card(\Gamma)}{\delta}}{\lambda}
  \Bigr],
\]
\State and set $P^\star_{PAC} \gets P_{\lambda^\star}$.

        \State \textbf{Outputs:} 
        \State $\lambda^\star$, $P^\star_{PAC}$: the optimizers of the minimization problem   
        \State $K \sim P^\star_{PAC}$: controller sampled from the learned posterior
    \end{algorithmic}
\end{algorithm}
In this finite setting, the resulting optimization problem is convex and can be solved efficiently using standard solvers such as \textit{CVX}~\citep{cvx}.

\subsection{Infinite controller space: PAC-Bayes controller learning algorithm}
Consider Case~2 shown in Section~\ref{sec:problem}, where $\mathcal{K} \subseteq \mathbb{R}^{d_u \times d_x}$. We draw $L'$ controllers $\{K_j\}_{j=1}^{L'}$ independently from the initial controller distribution $P$, and use each sampled controller $K_j$ to generate a trajectory dataset $\mathcal{S} =\{X_{K_j}^{(i)}\}^{i=1,\cdots ,n}_{j=1,\cdots,L'}$ according to the data-generating process described in Section~\ref{sec:thm}. 
In contrast to the finite-controller case, where $P$ can be optimized directly, the parameterized distribution $P$ is optimized by SGD using stochastic descent directions of the PAC-Bayes bound. In our SGD procedure, the distribution $P$ is parameterized by $\theta$ (e.g., mean, variance, or interval bounds), and optimization is performed over $\theta$ to update $P$ accordingly, hereafter denoted by $P_\theta$. Note that this is not a standard SGD, since the training objective does not admit a closed-form gradient and can only be evaluated via Monte Carlo sampling. Consequently, 
the descent direction must be inferred from these Monte Carlo estimates. That is, the distribution $P_\theta$ is then replaced with an updated one, and this procedure is repeated iteratively so that the training objective is progressively minimized. Let $\Phi(P_\theta)$ denote the training objective, defined as the right-hand side of inequality~\eqref{eq:thm2-mc} with $B_{cost}$ replaced by $\hat B_{cost}$. For $K\sim P_\theta$, consider $\mathcal{S}_K = \{X_{K}^{(i)}\}_{i=1}^{n} \subseteq \mathcal{S}$ for the trajectories generated under $K$. Define
\begin{align}\label{eq:phi}
\phi(\mathcal{S}_K)
\;:=\;
&
\hat C(K)
\;+\;
C_{\max}(\mathcal{S})\,\sqrt{\frac{1}{2L'}\,\ln\!\frac{2}{\delta'}}
\;+\;
\frac{\lambda\,\hat B_{cost}^2}{8n}\notag\\
&\;+\;
\frac{KL(P_\theta\Vert P_0)\;+\;\ln\!\frac{card(\Gamma)}{\delta}}{\lambda}.
\end{align}
The learning procedure is presented in Algorithm~\ref{alg:PAC-Bayes controller learning for time-invariant linear discrete-time system-infinite}.
\begin{algorithm}
\caption{PAC-Bayes controller learning for infinite controller space}\label{alg:PAC-Bayes controller learning for time-invariant linear discrete-time system-infinite}
\begin{algorithmic}[1]

\State \textbf{Inputs:}
\begin{itemize}
  \item $\theta_0$: initial posterior distribution parameters 
  \item $\delta,\delta'\in(0,1)$: confidence parameters 
  \item $P_0$: prior distribution over controller space $\mathcal K$
  \item Step size $\eta>0$, smoothing parameter $h>0$
  \item ${\mathrm{Iter}}$: number of iterations 
  \item $L'$: number of controllers (for Monte Carlo over $K\sim P_\theta$)
  \item $\Gamma$: a finite set of candidate values for the parameter $\lambda$
\end{itemize}
\State \textbf{Learning Stage:}
\State Initialize $\theta \gets \theta_0$

\For{$i = 1$ to ${\mathrm{Iter}}$}
  \State Sample random perturbation $\Delta \sim \mathcal{N}(0,I_{\dim(\theta)})$ 
  \State Construct perturbed parameters $\theta' \gets \theta + h\,\Delta$
  \State Monte Carlo sampling of controllers:
  \[
     \{K_j\}_{j=1}^{L'} \overset{\text{i.i.d.}}{\sim} P_\theta, 
     \qquad
     \{K'_j\}_{j=1}^{L'} \overset{\text{i.i.d.}}{\sim} P_{\theta'}
  \]

  \For{$j = 1$ to $L'$}
     \State Generate trajectories for controller $K_j$:
     \[
     \mathcal{S}_{K_j} = \{X_{K_j}^{(i)}\}_{i=1}^{n},
     \qquad
     \mathcal{S}_{K'_j} = \{X_{K'_j}^{(i)}\}_{i=1}^{n}
     \]
     according to the data-generating process described in Section~\ref{sec:thm}.
  \EndFor
  \State Compute empirical $\hat B_{cost}$ according to Eq.~\eqref{eq: empircal Bcost}
  \State Evaluate the upper-bound objective (Monte Carlo averages) according to Eq.~\eqref{eq:phi}:
  \[
  \Phi(P_\theta) \gets \frac{1}{L'} \sum_{j=1}^{L'} \phi(\mathcal{S}_{K_j}),
  \qquad
  \Phi(P_{\theta'}) \gets \frac{1}{L'} \sum_{j=1}^{L'} \phi(\mathcal{S}_{K'_j})
  \]

  \State Estimate the stochastic gradient:
  \[
  \hat{G}_\theta \gets \frac{\Phi(P_{\theta'}) - \Phi(P_\theta)}{h}\,\Delta
  \]

  \State Gradient descent update:
  \[
  \theta \gets \theta - \eta\,\hat{G}_\theta
  \]
\EndFor

\State \textbf{Outputs:}
\begin{itemize} \item $P_\theta$: learned distribution  over $\mathcal K$ 
\item $K \sim P_\theta$: controller sampled from the learned posterior
\end{itemize}
\end{algorithmic}
\end{algorithm}

\section{Numerical Experiments}
\label{sec:experiment}


\textbf{Example 1} (Finite controller space -- bound verification).
The time horizon is $T = 20$, the initial state is $x_0 = 0 \in \mathbb{R}^2$, and
$Q = I_2$ and $R = 0.1 I_1$.
The system matrices are drawn from truncated Gaussian distributions with
$a_1 = -0.3$, $a_2 = 0.3$ for $A$ and $b_1 = -0.3$, $b_2 = 0.3$ for $B$.
The system matrices $\mu_A$ and $\mu_B$  have entries taking values in the intervals $[-0.3,0.3]$.
Each realization of the system matrices is obtained by sampling its entries from this truncated Gaussian distribution, and the corresponding standard deviations of the entries lie in the range $[0, 0.1]$.
The sub-Gaussian parameter for the process noise is set to $\sigma_w = 0.5$.
The process noise is modeled as $w_t \sim \mathcal N(0,\Sigma_w)$ for
$t = 0,\dots,T-1$, where $\Sigma_w \in \mathbb{R}^{d_x \times d_x}$ is a
covariance matrix, and the corresponding standard deviations of the entries of $w_t$ lie in the range $[0.4,\sigma_w]$.
Consider controllers of the form $K = [k_1,k_2]$ and define a finite controller space
$
\mathcal{K} = \{K_1,\dots,K_{25}\}, \quad K_i = [k_{1,i}\;k_{2,i}],
$
where $\{k_{1,i}\}$ and $\{k_{2,i}\}$ are chosen as 5 uniformly spaced points from $[0,0.3]$ and $[-0.6,-0.3]$, respectively.
The candidate set of $\lambda$ is $\Gamma=\{2.85,3.76,4.94,6.51,8.56\}$.
The prior $P_0$ over $\mathcal{K}$ is uniform.
The confidence parameter is $\delta = 0.05$.
For the number of trajectories per controller, we take
$
n \in \{10,20,\cdots,100\},
$
and for each $n$ we apply Algorithm~\ref{alg:PAC-Bayes controller learning for time-invariant linear discrete-time system-finite}.
The expected cost is estimated using an independent test set of 100 trajectories per controller.
\begin{figure}[htbp]
  \centering
\includegraphics[width=0.5\textwidth]{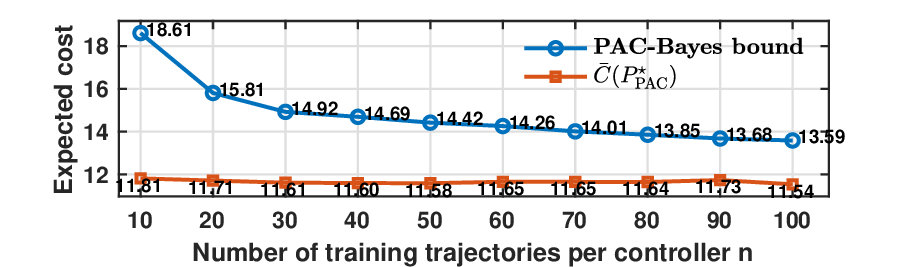} 
  \caption{\footnotesize Comparison of PAC-Bayes upper bounds and expected cost, across varying training trajectories per controller, for a time-invariant linear discrete-time system with a finite controller space.}
  \label{fig:ex1-bound verification}
\end{figure}

Figure~\ref{fig:ex1-bound verification} validates the PAC-Bayes bound.
The expected costs of learned $P^\star_{PAC}$ remain below the corresponding PAC-Bayes upper bound for all $n$, and the bound becomes tighter as $n$ increases. In this example, even a small number of training trajectories, for example $n=10$ per controller, already yields a learned distribution $P_{\mathrm{PAC}}^\star$ that works well on the test trajectories.
Thus, the expected cost changes only mildly as $n$ increases, since larger $n$ mainly tightens the PAC-Bayes bound rather than changing the learned posterior.
For higher-dimensional systems, larger $n$ may help learn a better posterior.

\textbf{Example 2} (Controller evaluation).
To further evaluate the controller learned by our PAC-Bayes approach, we consider a modified version of Example~1 in which the classical finite-horizon LQG controller is globally optimal for the true underlying system.
The same basic setup as in Example~1 is reused, with only the system and the reference controller changed.
The system matrices $\mu_A$ and $\mu_B$ have entries taking values in the intervals $[-2,2]$.
To ensure that the LQG controller is globally optimal for the true dynamics, the standard deviations of the parameter distributions are set to zero, so that $A = \mu_A$ and $B = \mu_B$ are fixed.
Given $(A,B)$, the finite-horizon LQG controller is computed and its expected cost is evaluated under the resulting time-varying gains $\{K(t)^{LQG}\}_{t=0}^{T-1}$, i.e.\ $u(t) = K(t)^{LQG} x_t$.
The corresponding finite-horizon LQG cost is $J(K(t)^{LQG};A,B)$, where $J$ is defined in Eq.~\eqref{eq:J_KAB}.
In this example, a finite controller space
$
\mathcal{K} = \{K_1,\dots,K_{25}\}, \quad
K_i = [k_{1,i}\;k_{2,i}],
$
is defined, where $\{k_{1,i}\}$ and $\{k_{2,i}\}$ are 5 uniformly spaced points in
$[0.75,1.25]$ and $[-1,-0.5]$, respectively.
The candidate set of $\lambda$ is
$
\Gamma = \{0.0956,\;0.0276,\; 0.0027,\; 0.0019,\;2.65\times 10^{-4},\; 1\times 10^{-4},\; 4.79\times 10^{-5}\}.
$
Our PAC-Bayes algorithm does not have access to $(A,B)$ and only uses the training trajectories to learn a posterior $P^\star_{PAC}$ over $\mathcal{K}$.
The goal is to check whether the learned $P^\star_{PAC}$ can achieve a closed-loop performance comparable to the LQG controller with full knowledge of $(A,B)$.

\begin{figure}[htbp]
  \centering
\includegraphics[width=0.5\textwidth]{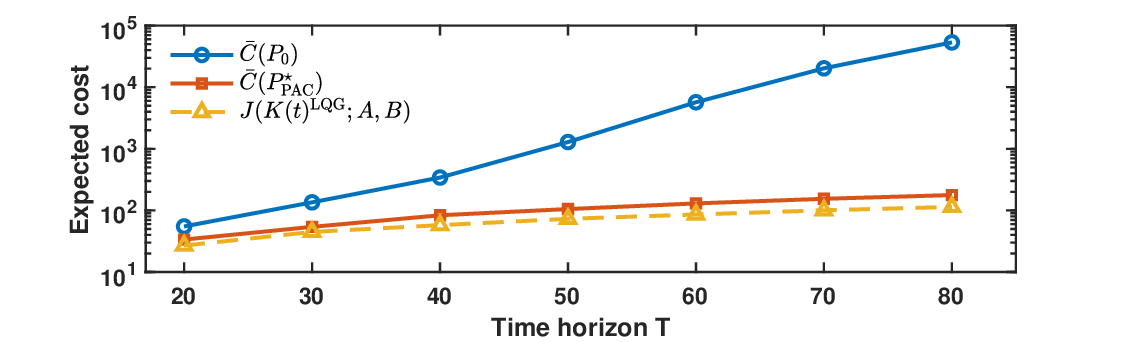} 
  \caption{\footnotesize Comparison of the expected costs of the prior $P_0$, the learned posterior $P^\star_{PAC}$, and the finite-horizon LQG controller, denoted by $\bar C(P_0)$, $\bar C(P^\star_{PAC})$, and $J(K(t)^{\mathrm{LQG}};A,B)$, respectively, as the horizon length $T$ varies with a fixed number of training trajectories per controller $n=10$. }
  \label{fig:ex2-varying T}
\end{figure}

Figure~\ref{fig:ex2-varying T} evaluates our learned controller by comparing its expected cost $\bar C(P^\star_{PAC})$ with those of the prior $P_0$ and the finite-horizon LQG controller. As the time horizon $T$ increases while the number of training trajectories per controller is fixed at $n=10$, the expected costs of all three controllers grow. The learned posterior $P^\star_{PAC}$ consistently achieves a substantial reduction in cost relative to the prior $P_0$, with the gap widening to several orders of magnitude for large time horizons. At the same time, $\bar C(P^\star_{PAC})$ remains very close to the LQG benchmark $J(K(t)^{\mathrm{LQG}};A,B)$ across all values of $T$. These results demonstrate that the PAC-Bayes learning procedure is highly effective: the learned controller distribution $P^\star_{PAC}$ achieves performance comparable to the LQG controller while clearly outperforming the prior.

\textbf{Example 3} (Infinite controller space -- performance improvement over the prior).
This example is modified from Example~1. Unless otherwise stated, the basic
finite-horizon setup and cost matrices are the same as in Example~1.
The truncation intervals for the system matrices are changed to
$a_1=-1$, $a_2=1$ for $A$ and $b_1=-1$, $b_2=1$ for $B$.
The mean matrices $\mu_A$ and $\mu_B$ are sampled entrywise from $[-1,1]$,
and the corresponding standard deviations are sampled from $[0,0.1]$.
The process-noise level is set to $\sigma_{w,0}=0.25$, and the entrywise
standard deviations of the process noise are sampled from
$[0.2,\sigma_{w,0}]$.
Instead of the finite controller space used in Example~1, we consider an infinite controller space with static gains $K\in\mathbb{R}^{1\times2}$.
Both the prior $P_0$ and posterior $P_\theta$ are modeled as product truncated Gaussian distributions over the entries of $K$, where each entry satisfies
$
K_{i}\sim\mathcal N(\mu_{i},\sigma_{i}^2)
$
truncated to $[L_{i},U_{i}]$.
The posterior parameter is
$
\theta=(\mu_K,\sigma_K,L_K,U_K).
$
The prior parameters are initialized as
$
\mu_K=[-0.5,-1.5],$
$\sigma_K=[0.25,0.25],
$
with support bounds
$
L_K=[-0.575,-1.625],$
$U_K=[0.25,-0.25].
$
The posterior is re-initialized from the same prior for each value of $n$, and the posterior learning is run for $10$ iterations.
For the number of training trajectories per controller, we take
$
n\in\{10,20,30,40,50\}.
$
For each $n$, we apply
Algorithm~\ref{alg:PAC-Bayes controller learning for time-invariant linear discrete-time system-infinite}
with $L'=10$, $\delta=0.5$, and $\delta'=0.25$.
After learning $P_\theta$, we estimate its expected cost by sampling $100$
controllers from the learned posterior and evaluating each controller using
$100$ test trajectories.
The fixed prior $P_0$ is evaluated in the same way and used as a reference.
\begin{figure}[htbp]
  \centering
\includegraphics[width=0.5\textwidth]{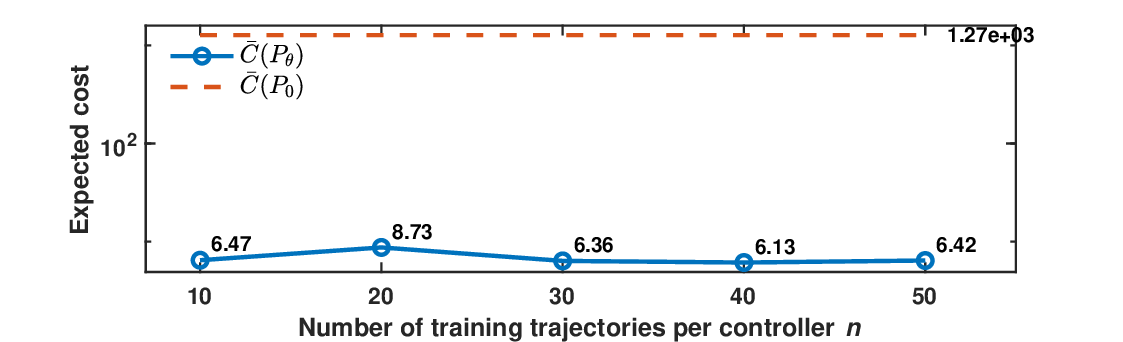} 
\caption{\footnotesize Comparison of expected costs under the prior $P_0$ and the learned posterior $P_\theta$ as the number of training trajectories per controller varies.}
  \label{fig:ex3}
\end{figure}

Figure~\ref{fig:ex3} compares the expected costs under the prior $P_0$ and the learned posterior $P_\theta$ for different numbers of training trajectories per controller $n$.
The prior has a much larger expected cost, around $1270$, whereas the learned posterior achieves costs approximately in the range of $6$--$9$.
Thus, the learned posterior substantially improves performance over the prior, reducing the expected cost by more than two orders of magnitude.
This indicates that the PAC-Bayes learning procedure effectively concentrates the posterior on low-cost controllers in the infinite controller space.
The mild fluctuations arise from the stochastic learning over $\theta=(\mu_K,\sigma_K,L_K,U_K)$, where only the posterior after 10 iterations is reported.
Nevertheless, the learned posterior already substantially outperforms the prior.

\section{Conclusion and Future Work}
\label{sec:conclusion}
In this paper, we present a PAC-Bayes framework for learning controllers under unknown linear dynamics. 
The method handles unmodified quadratic costs, and is numerically applied for both finite and infinite controller spaces. When LQG is optimal, our method achieves comparable performance. In future work, we plan to develop adaptive algorithms to automatically identify suitable controller spaces. 


\bibliography{ifacconf}

\end{document}